\documentclass[11pt]{article}

\usepackage{amssymb,amsmath,amsfonts,amsthm}
\usepackage{latexsym}
\usepackage{graphics}
\usepackage{indentfirst}
\usepackage{hyperref}
\usepackage[capitalise, noabbrev, nameinlink]{cleveref}
\usepackage{comment}
\usepackage[shortlabels]{enumitem}
\usepackage[english]{babel}
\usepackage{tikz}
\usetikzlibrary{fit,arrows.meta,backgrounds}
\usetikzlibrary{shapes,decorations,graphs,graphs.standard,quotes,backgrounds}
\usetikzlibrary{decorations.markings}
\usepackage{esdiff}
\usepackage{thmtools}
\usepackage{thm-restate}

\definecolor{azure(colorwheel)}{rgb}{0.0, 0.5, 1.0}
\definecolor{Pink}{RGB}{255, 105, 180}

\usetikzlibrary{arrows.meta}

\newtheorem{innercustomthm}{Theorem}
\newenvironment{customthm}[1]
  {\renewcommand\theinnercustomthm{#1}\innercustomthm}
  {\endinnercustomthm}

\newtheorem*{thm*}{Theorem}
\newtheorem{thm}{Theorem}
\newtheorem{lem}[thm]{Lemma}
\newtheorem{pro}[thm]{Proposition}
\newtheorem{obs}[thm]{Observation}

\newtheorem{ques}[thm]{Question}

\crefname{thm}{Theorem}{Theorems}
\crefname{lem}{Lemma}{Lemmas}

\newcommand{\N}{\mathbb{N}}

\newcommand{\R}{\mathbb{R}}

\allowdisplaybreaks

\begin{document}

\title{An Ohba-like Result for Flexible List Coloring}

\author{ 
Michael C. Bowdoin \thanks{Mitchell College of Business, University of South Alabama, Mobile, AL, USA (mb2139@jagmail.southalabama.edu)}
\and
Yanghong Chi  \thanks{Alabama School of Math and Science, Mobile, AL, USA (chiyanghong12138@gmail.com)}
\and 
Christian B. Ellington \thanks{School of Computing, University of South Alabama, Mobile, AL, USA (cbe2222@jagmail.southalabama.edu)}
\and
Bella Ives  \thanks{Alabama School of Math and Science, Mobile, AL, USA (Logan.Ives@asms.net)}
\and 
Seoju Lee   \thanks{Alabama School of Math and Science, Mobile, AL, USA (Seojuu.lee@gmail.com)}
\and
Fennec Morrissette \thanks{School of Computing, University of South Alabama, Mobile, AL, USA (ajm2228@jagmail.southalabama.edu )}
\and
Jeffrey A. Mudrock \thanks{Department of Mathematics and Statistics, University of South Alabama, Mobile, AL, USA (mudrock@southalabama.edu)}
}

\maketitle

\begin{abstract}
Chromatic-choosability is a notion of fundamental importance in list coloring.  A graph $G$ is \emph{chromatic-choosable} when its chromatic number, $\chi(G)$, is equal to its list chromatic number $\chi_{\ell}(G)$.  Flexible list coloring was introduced by Dvo\v{r}\'{a}k, Norin, and Postle in 2019 in order to address a situation in list coloring where we still seek a proper list coloring, but each vertex may have a preferred color assigned to it, and for those vertices we wish to color as many of them with their preferred colors as possible.  In flexible list coloring, the \emph{list flexibility number} of $G$, denoted by $\chi_{\ell flex}(G)$, serves as the natural analogue of $\chi_{\ell}(G)$.  Specifically, $\chi_{\ell flex}(G)$ is the minimum $k$ such that $G$ is $(k,1/\rho(G))$-flexible where $\rho(G)$ is the Hall ratio of $G$.  In 2002, Ohba famously showed that for any graph $G$, there exists an $N \in \N$ such that $\chi(K_p \vee G) = \chi_{\ell}(K_p \vee G)$ whenever $p \geq N$.  Since $\chi(G) \leq \chi_{\ell}(G) \leq \chi_{\ell flex}(G)$, it is natural to ask whether this result holds if $\chi_{\ell}$ is replaced with $\chi_{\ell flex}$.  In this paper we not only show that this result doesn't hold in general if $\chi_{\ell}$ is replaced with $\chi_{\ell flex}$, but we also give a characterization of the graphs for which it does hold. 
 
\medskip



\end{abstract}

\section{Introduction}\label{intro}
In this paper, all graphs are nonempty, finite, simple graphs.  Generally speaking, we follow West~\cite{W01} for terminology and notation.  The set of natural numbers is $\N = \{1,2,3, \ldots \}$.  For $m \in \N$, we write $[m]$ for the set $\{1, \ldots, m \}$.  We also take $[0]$ to be the empty set. For a graph $G$, $V(G)$ and $E(G)$ denote the vertex set and the edge set of $G$ respectively. We use $\alpha(G)$ and $\omega(G)$ for the independence number of $G$ and clique number of $G$ respectively.  When $u$ and $v$ are adjacent in $G$, we write $uv$ or $vu$ for the edge with endpoints $u$ and $v$.  By $H \subseteq G$, we mean that $H$ is a subgraph of $G$. If $S \subseteq V(G)$, $G[S]$ denotes the subgraph of $G$ induced by $S$, and $G-S$ is the graph $G[V(G)-S]$.  For any $v \in V(G)$, we let $G-v$ be the graph $G[V(G)-\{v\}]$. For any $E \subseteq E(G)$, we let $G-E$ be the subgraph of $G$ with vertex set $V(G)$ and edge set $E(G)-E$.  Also, when $G$ and $H$ are vertex disjoint graphs we write $G \vee H$ for the join of $G$ and $H$, and when $G= K_1$ we say that $G \vee H$ is the \emph{cone} of $H$.

Our goal in this paper is to extend results on chromatic-choosability to the context of flexible list coloring, a recently introduced variation on list coloring. We begin by reviewing list coloring and chromatic-choosability.

\subsection{List Coloring and Chromatic-Choosability} \label{basic}

For classical vertex coloring of graphs, the goal is to color the vertices of a graph $G$ with up to $m$ colors from $[m]$ so that adjacent vertices receive different colors, a so-called \emph{proper $m$-coloring}.  The \emph{chromatic number} of a graph $G$, denoted by $\chi(G)$, is the smallest $m$ such that $G$ has a proper $m$-coloring.  List coloring is a well-known variation on classical vertex coloring that was introduced independently by Vizing~\cite{V76} and Erd\H{o}s, Rubin, and Taylor~\cite{ET79} in the 1970s.  For list coloring, we associate a \emph{list assignment} $L$ with a graph $G$ such that each vertex $v \in V(G)$ is assigned a list of available colors $L(v)$ ($L$ is said to be a list assignment for $G$).  We say $G$ is \emph{$L$-colorable} if there is a proper coloring $f$ of $G$ such that $f(v) \in L(v)$ for each $v \in V(G)$ (we refer to $f$ as a \emph{proper $L$-coloring} of $G$).  A list assignment $L$ for $G$ is called a \emph{$k$-assignment} if $|L(v)|=k$ for each $v \in V(G)$.  A graph $G$ is said to be \emph{$k$-choosable} if it is $L$-colorable whenever $L$ is a $k$-assignment for $G$.  The \emph{list chromatic number} of a graph $G$, denoted by $\chi_\ell(G)$, is the smallest $m$ for which $G$ is $m$-choosable.  It is easy to show that for any graph $G$, $\chi(G) \leq \chi_\ell(G)$.   

A graph $G$ is called \emph{chromatic-choosable} if $\chi(G) = \chi_{\ell}(G)$~\cite{O02}.   Determining whether a graph is chromatic-choosable is, in general, a challenging problem.  Perhaps the most famous conjecture involving list coloring is about chromatic-choosability.  Indeed, the Edge List Coloring Conjecture states that every line graph of a loopless multigraph is chromatic-choosable (see~\cite{HC92}). Although the conjecture remains open in general, some progress has been made. Galvin~\cite{G95} proved the conjecture for bipartite multigraphs while Kahn~\cite{K00} showed that the conjecture holds asymptotically. These results illustrate a broader theme in list coloring: determining structural conditions that guarantee chromatic-choosability.

In 2002, Ohba showed the following important result which served as the inspiration for the questions studied in this paper.

\begin{thm}[\cite{O02}] \label{thm: OhbaO} For any graph $G$, there exists an $N \in \N$ such that $K_p \vee G$ is chromatic-choosable whenever $p \geq N$.  
\end{thm}

In the same paper, Ohba conjectured the considerably stronger statement that every graph $G$ satisfying $|V(G)|\leq 2\chi(G)+1$ is chromatic-choosable. This conjecture attracted considerable attention and motivated a number of partial results. In particular, Reed and Sudakov~\cite{RS04} obtained strong results for graphs whose chromatic number is close to their order. Finally, in 2015 Noel, Reed, and Wu~\cite{NR15} proved Ohba's conjecture in full.

\begin{thm}[\cite{NR15}] \label{thm: Ohba}
If $G$ is a graph satisfying $|V(G)| \leq 2\chi(G) + 1$, then $G$ is chromatic-choosable.
\end{thm}

\subsection{Flexible List Coloring} 

Flexible list coloring was introduced by Dvo\v{r}\'{a}k, Norin, and Postle~\cite{DN19} in 2019, and since its introduction, it has been studied by many authors (see e.g.,~\cite{BB24, BMS22, CCM22, KM23, LMZ22, TM19}).  Flexible list coloring addresses a situation in list coloring where we still seek a proper list coloring, but each vertex may have a preferred color assigned to it, and for those vertices we wish to color as many of them with their preferred colors as possible.  Specifically, suppose $G$ is a graph and $L$ is a list assignment for $G$.  A \emph{request of $L$} is a function $r$ with nonempty domain $D\subseteq V(G)$ such that $r(v) \in L(v)$ for each $v \in D$.  For each $v \in D$ we say that the pair $(v,r(v))$ is the \emph{vertex request for $v$ with respect to $r$} (we omit the phrase with respect to $r$ when $r$ is clear from context).  For any $\epsilon \in [0,1]$, the triple $(G,L,r)$ is \emph{$\epsilon$-satisfiable} if there exists a proper $L$-coloring $f$ of $G$ such that $f(v) = r(v)$ for at least $\epsilon|D|$ vertices in $D$.  Additionally, when $f(v) = r(v)$ for some $v \in D$, we say $f$ \emph{satisfies} the vertex request for $v$, and when $f(v) = r(v)$ for at least $\epsilon|D|$ vertices in $D$, we say $f$ \emph{satisfies at least $\epsilon|D|$ of the vertex requests of $r$}.

The pair $(G,L)$ is \emph{$\epsilon$-flexible} if $(G,L,r)$ is $\epsilon$-satisfiable whenever $r$ is a request of $L$.  We say that $G$ is \emph{$(k, \epsilon)$-flexible} if $(G,L)$ is $\epsilon$-flexible whenever $L$ is a $k$-assignment for $G$.  Note that if $G$ is $k$-choosable, then $G$ is $(k,0)$-flexible.

It is worth mentioning that requiring a positive proportion of vertex requests to be satisfied can be substantially more difficult than ordinary list coloring. For example, every $d$-degenerate graph is $(d+1)$-choosable. However, Dvo\v{r}\'ak, Norin, and Postle~\cite{DN19} asked whether, for each $d\geq 1$, there exists an $\epsilon_d>0$ such that every $d$-degenerate graph is $(d+1,\epsilon_d)$-flexible. This question remains open; in fact, it is unknown whether there exists an $\epsilon_2>0$ such that every $2$-degenerate graph is $(3,\epsilon_2)$-flexible~\cite{BB25}. Even the much weaker question of whether one can always satisfy a single vertex request of an arbitrary $(d+1)$-assignment for a $d$-degenerate graph is not known in general~\cite{DN19}. 

For a graph $G$, it is natural to ask: what is the largest $\epsilon$ so that $G$ is $(k,\epsilon)$-flexible for some $k \in \N$?  For a $k$-assignment $L$ for $G$ and request $r$ of $L$ with domain $D$, it is possible that $r(v)$ is the same color for all $v\in D$.  Then at most $\alpha(G[D])$ vertices in $D$ will have their request satisfied in a proper $L$-coloring.  So, $\epsilon$ must satisfy $\epsilon |D|\le \alpha(G[D])$ when $G$ is $(k,\epsilon)$-flexible.  

Thus, $\epsilon\le \min_{\emptyset\not=D\subseteq V(G)} \alpha(G[D])/|D|$ for any $(k,\epsilon)$-flexible graph $G$.  The \emph{Hall ratio} of a graph $G$, denoted by $\rho(G)$, is given by 
$$\rho(G) = \max_{\emptyset\not= H \subseteq G} \frac{|V(H)|}{\alpha(H)}.$$  
The Hall ratio was first studied in 1997 by Hilton, Johnson Jr., and Leonard~\cite{HJ97} under the name fractional Hall-condition number.  It is shown in~\cite{HJ97} that for any graph $G$, $\omega(G) \leq \rho(G) \leq \chi_f(G)$ where $\chi_f(G)$ is the fractional chromatic number of $G$.  Note that among subgraphs $H$ of $G$ with fixed vertex set $D$, $\alpha(H)$ is minimized when $H=G[D]$.  Therefore, $\min_{\emptyset\not=D\subseteq V(G)} \alpha(G[D])/|D| = 1/\rho(G)$.  Moreover, this bound on feasible $\epsilon$ for $(k,\epsilon)$-flexibility is attainable.

\begin{pro} [\cite{KM23}] \label{pro: fundamental}
There exists a $k$ such that $G$ is $(k,\epsilon)$-flexible if and only if $\epsilon\le 1/\rho(G)$.
\end{pro}

\cref{pro: fundamental} naturally leads to the following graph invariant which has received some attention in the literature (see~\cite{BB24, CH25, KM23}). The \emph{list flexibility number} of $G$, denoted by $\chi_{\ell flex}(G)$, is the smallest $k$ such that $G$ is $(k,1/\rho(G))$-flexible.  It is easy to show that 
$$\chi(G) \leq \chi_{\ell}(G) \leq \chi_{\ell flex}(G) \leq \Delta(G)+1.$$
Since the list flexibility number of a graph is an upper bound on its list chromatic number, it is natural to ask whether the analogue of Ohba's 2002 result (\cref{thm: OhbaO}) holds in the context of list flexibility.
\begin{ques}\label{ques: BIGDADDY}
For every graph $G$, is there an $N\in \N$ such that $\chi_{\ell flex}(K_p\vee G) = \chi(K_p\vee G)$ whenever $p \geq N$?
\end{ques}
In this paper we not only show the answer to \cref{ques: BIGDADDY} is no, but we also give a characterization of the graphs for which the answer is yes.  Interestingly, it is easy to show: if $\chi_{\ell}(G) = \chi(G)$, then $\chi_{\ell}(K_1 \vee G) = \chi(K_1 \vee G)$.  In words, this says that the cone of a chromatic-choosable graph is chromatic-choosable.  To our surprise, we have not been able to answer the following question.

\begin{ques}\label{ques: smallDADDY}
For every graph $G$ satisfying $\chi_{\ell flex}(G) = \chi(G)$, does it follow that $\chi_{\ell flex}(K_1\vee G) = \chi(K_1\vee G)$?
\end{ques}

\subsection{Outline of Paper}

In \cref{cone} we make some progress on \cref{ques: smallDADDY}.  We begin by proving a result on the list epsilon flexibility function of the cone of a graph.  The \emph{list epsilon flexibility function} of $G$, denoted by $\epsilon_\ell(G,k)$, is the function that maps each $k \geq \chi_{\ell}(G)$ to the largest $\epsilon \in [0,1]$ such that $G$ is $(k,\epsilon)$-flexible.  Clearly, $\epsilon_{\ell}(G,k)=a/b$ for some integers $0\le a\le b\le |V(G)|$ and $\epsilon_\ell(G,k) \leq 1/\rho(G)$ (see~\cite{BB24, KM23} for more properties).

\begin{thm} \label{thm: join}
Suppose $q$ is a positive rational number and $k$ is a positive integer satisfying $k\geq q$. Suppose $G$ is a graph such that $\epsilon_{\ell}(G,k) \geq 1/q$. If $G' = K_1\vee G$, then $1/(q+1) \leq \epsilon_{\ell}(G',k+1)$.
\end{thm}

\cref{thm: join} provides a general mechanism for transferring flexibility guarantees from a graph to its cone. This allows us to answer \cref{ques: smallDADDY} in the affirmative for every graph whose Hall ratio is equal to its clique number.

\begin{thm} \label{cor: hallin}
Suppose $G$ is such that $\rho(G)=\omega(G)$ and $\chi_{\ell flex}(G) = k$. Then $\chi_{\ell flex}(K_1 \vee G) \leq k+1$. Consequently, if $G$ also satisfies $\chi_{\ell flex}(G) = \chi(G)$, then $\chi_{\ell flex}(K_1 \vee G) = \chi(K_1 \vee G).$
\end{thm}

Note that for any graph $G$ for which $\omega(G)$ equals the fractional chromatic number of $G$ (e.g., perfect graphs), we have $\rho(G) = \omega(G)$.  On the other hand, \cref{cor: hallin} doesn't apply when $G=C_{2k+1}$ and $k \geq 2$ since in this case $\omega(G)=2 < (2k+1)/k = \rho(G)$.  We end \cref{cone} by showing that the answer to \cref{ques: smallDADDY} is yes when $G$ is an odd cycle.  This provides further evidence towards an affirmative answer to \cref{ques: smallDADDY}.

In \cref{answer} we show that the answer to \cref{ques: BIGDADDY} is no and characterize the graphs for which the answer is yes by proving the following result.

\begin{thm} \label{thm: main}
For every graph $G$, there is an $N\in\mathbb{N}$ such that $\chi_{\ell flex}(K_p\vee G)=\chi(K_p\vee G)$ whenever $p\geq N$ if and only if, for every $S\subseteq V(G)$ with $|S|=\omega(G)+1$, there is a proper $\chi(G)$-coloring of $G$ that uses at most $\omega(G)=|S|-1$ colors on $S$.
\end{thm}

While many graphs satisfy the condition in \cref{thm: main}, it is challenging to find graphs that do not; nevertheless, there are infinitely many graphs for which the condition in \cref{thm: main} fails to hold.  Recall a graph is said to be \emph{uniquely $k$-colorable} if there is only one partition of its vertex set into $k$ color classes.  It is well-known that for each $k \geq 3$ there is a uniquely $k$-colorable graph $G$ satisfying $\omega(G) = k-1$ (see~\cite{HH69}).  Note that for such a graph $G$ if $S \subseteq V(G)$ is a $k$-element set containing one vertex from each color class of a proper $k$-coloring of $G$, there do not exist two vertices in $S$ that are colored the same by some proper $k$-coloring of $G$. 

The existence of graphs that don't satisfy the condition of \cref{thm: main} also demonstrates that for any $\epsilon \in \R^+$ there is a graph $G$ such that $\chi_{\ell flex}(G) > \chi(G)$ and $|V(G)|/\chi(G) < 1 + \epsilon$.  Consequently, one could not hope to prove a flexible list analogue of \cref{thm: Ohba} without additional hypotheses.

On the other hand, for a graph $G$ that does satisfy the condition of \cref{thm: main} (e.g., a bipartite graph, a complete multipartite graph, a uniquely $k$-colorable graph with clique number $k$) no serious attempt in this paper has been made to find the optimal $N$ for which  $\chi_{\ell flex}(K_p \vee G) = \chi(K_p \vee G)$ whenever $p \geq N$.  Consequently, for a future research project it would be interesting to study the smallest such $N$.  The analogous problem of finding the smallest possible $N$ for which  $\chi_{\ell}(K_p \vee G) = \chi(K_p \vee G)$ whenever $p \geq N$ has been pursued by several authors~\cite{A09, N14, AJ15, CG24}. 

\section{List Flexibility of the Cone of a Graph} \label{cone}

We begin this section by proving \cref{thm: join}.

\begin{proof}
Suppose the vertex set of the copy of  $K_1$ used to form $G'$ is $\{u\}$. Suppose $L$ is an arbitrary ($k+1$)-assignment of $G'$ and $r$ is a request of $L$ with domain $D$. Let $U=\{u\} \cap D$, and let $C : L(u) \rightarrow \mathbb{Z}$ be the function given by $C(z) = |r^{-1}(z) \cap V(G)|$. Name the colors of $L(u)$ so that $L(u) = \{c_1,\ldots,c_{k+1}\}$ and $C(c_1) \geq C(c_2) \geq \cdots \geq C(c_{k+1})$. Let $R=D- r^{-1}(L(u))$.  Notice
$$|D| =|U| + |R| + \sum_{i=1}^{k+1} C(c_i).$$
We will now construct two proper $L$-colorings of $G'$, $f$ and $g$.  We construct $f$ as follows. Let $f(u) = c_{k+1}$. For each $v\in V(G)$, let $L'(v)=L(v)-\{c_{k+1}\}$. If $L'(v)$ has $k+1$ elements and $v\notin D$, delete an element of $L'(v)$ from $L'(v)$.  If $L'(v)$ has $k+1$ elements and $v\in D$, delete an element of $L'(v)$ from $L'(v)$ that is not $r(v)$.  Note that $L'$ is a $k$-assignment for $G$. Let $D' = D- (r^{-1}(c_{k+1}) \cup U)$, and let $r': D'\rightarrow \bigcup_{v\in V(G)}L'(v)$ be the function given by $r'(v) = r(v)$.  If $D'=\emptyset$, let $h$ be a proper $L'$-coloring of $G$ which exists since $G$ is $k$-choosable. If $D'\neq\emptyset$, note that $r'$ is a request of $L'$. Since $\epsilon_{\ell}(G,k) \geq 1/q$, there is a proper $L'$-coloring $h$ of $G$ with the property that $h(v) = r'(v)$ for at least $|D'|/q$ of the vertices in $D'$. Thus, in either case, there is a proper $L'$-coloring $h$ of $G$ that agrees with $r'$ on at least $|D'|/q$ of the vertices in $D'$.  To complete $f$, let $f(v) = h(v)$ for each $v\in V(G)$. Notice $f$ satisfies at least $|D'|/q$ of the vertex requests of $r$, and
$|D'| = |R| + \sum_{i=1}^k C(c_i).$

We construct $g$ as follows. If $u\notin D$, let $g(u)=c_1$.  If $u\in D$, let $g(u) = r(u)$.  For each $v\in V(G)$, let $L''(v)=L(v)-\{g(u)\}$. If $L''(v)$ has $k+1$ elements and $v\notin D$, delete an element of $L''(v)$ from $L''(v)$.  If $L''(v)$ has $k+1$ elements and $v\in D$, delete an element of $L''(v)$ from $L''(v)$ that is not $r(v)$.  Note that $L''$ is a $k$-assignment for $G$. Let $D'' = D- (r^{-1}(g(u)) \cup U)$, and let $r'': D''\rightarrow \bigcup_{v\in V(G)}L''(v)$ be the function given by $r''(v) = r(v)$. If $D''=\emptyset$, let $j$ be a proper $L''$-coloring of $G$, which exists since $G$ is $k$-choosable. If $D''\neq\emptyset$, note that $r''$ is a request of $L''$. Since $\epsilon_{\ell}(G,k) \geq 1/q$, there is a proper $L''$-coloring $j$ of $G$ that satisfies at least $|D''|/q$ of the vertex requests of $r''$. Thus, in either case, there is a proper $L''$-coloring $j$ of $G$ that agrees with $r''$ on at least $|D''|/q$ vertices of $D''$. To complete $g$, let $g(v) = j(v)$ for each $v\in V(G)$. Notice $g$ satisfies at least $|U|+|D''|/q$ of the vertex requests of $r$, and 
$|D''| \geq |R| + \sum_{i=2}^{k+1} C(c_i).$

Now, notice that if $|D'|/q \geq |D|/(q+1)$, then the desired result follows as $f$ satisfies at least $|D|/(q+1)$ of the vertex requests of $r$.  So, we may assume that $|D'|/q < |D|/(q+1)$. This means $(\sum_{i=1}^{k} C(c_i) +|R|)/q < (|U|+|R|+\sum_{i=1}^{k+1}C(c_i))/(q+1)$ which implies $$\sum_{i=1}^{k} C(c_i) +|R| < q|U|+ qC(c_{k+1}).$$ Since $|R| \geq 0$, we have $\sum_{i=1}^{k} C(c_i) < q|U|+ qC(c_{k+1})$ which implies
$$\sum_{i=1}^{k}\left(C(c_i)-\frac{q}{k}C(c_{k+1})\right) < q|U|.$$
Notice $k \geq q$ implies $qC(c_{k+1})/k \leq C(c_{k+1}) \leq C(c_i)$ for each $i \in [k]$. Consequently, for each $i \in [k]$, $C(c_i)-qC(c_{k+1})/k \geq 0.$ Thus, $$\frac{1}{q}C(c_1)-\frac{1}{k}C(c_{k+1}) < |U|.$$
Clearly, this inequality is a contradiction when $|U| = 0$. Therefore, we may assume from this point forward that $|U| = 1$.

Notice that $k \geq q$ implies $k/q^2 \geq 1/k$ which implies $C(c_1)/q-C(c_{k+1})/k \geq C(c_1)/q - kC(c_{k+1})/q^2$. Since $C(c_1)/q-C(c_{k+1})/k < 1$, $C(c_1)/q - kC(c_{k+1})/q^2 \leq 1$. This implies $$\frac{q}{q+1} + \frac{1}{q(q+1)}kC(c_{k+1}) \geq \frac{1}{q+1}C(c_1).$$ 
\\
Note that $\sum_{i=2}^{k+1} C(c_i) \geq kC(c_{k+1})$. As such, $$\frac{q}{q+1} + \left(\frac{1}{q(q+1)}\right) \sum_{i=2}^{k+1}C(c_i) \geq \frac{1}{q+1}C(c_1).$$ 
\\
Note also that $(1/q\ - 1/(q+1))|R| \geq 0$. So, $$1-\frac{1}{q+1} + \left(\frac{1}{q}\ - \frac{1}{q+1}\right) \sum_{i=2}^{k+1}C(c_i)+ \left(\frac{1}{q}\ - \frac{1}{q+1}\right)|R|\geq \frac{1}{q+1}C(c_1).$$ Thus, $$|U| + \frac{1}{q} \left(\sum_{i=2}^{k+1} C(c_i) + |R|\right) \geq \frac{1}{q+1} \left(|U| + \sum_{i=1}^{k+1} C(c_i) + |R|\right) = \frac{1}{q+1}|D|.$$

Notice the left side of the most recent inequality is a lower bound on the number of vertex requests of $r$ satisfied by $g$. Thus we have that $f$ or $g$ is a proper $L$-coloring of $G'$ that satisfies at least $|D|/(q+1)$ vertex requests of $r$.
\end{proof}

Our proofs of \cref{cor: hallin} and \cref{thm: main} require some results and terminology from~\cite{B16} which we now present. For a fixed graph $G$, the \emph{$w$-function of $G$}, denoted by $w(a,G)$, maps each $a \in [\alpha(G)]$ to the order of the largest subgraph of $G$ that has independence number $a$. Several basic statements follow immediately from the definition.

\begin{pro} [\cite{B16}] \label{pro: Basic}
For a fixed graph $G$, the following statements hold.
\\
1. $w(1,G) = \omega(G).$
\\
2. $w(\alpha(G), G) = |V(G)|.$
\\
3. $\rho(G) = \max_{a \in [\alpha(G)]} w(a,G)/a.$
\end{pro}

The following result is important for this paper in the case where $H$ is a complete graph.

\begin{thm} [\cite{B16}] \label{thm: aTheorem}
Suppose $G$ and $H$ are graphs such that $\alpha(H) \leq \alpha(G)$. Then for each $a \in [\alpha(G)]$, 
\[w(a, H \vee G)=
        \begin{cases}
        w(a,H)+w(a,G) \text{ if } a\in [\alpha(H)]
        \\
        |V(H)| +w(a,G) \text{ if } \alpha(H) \leq a \leq \alpha(G).
        
        \end{cases}
\]
\end{thm}

We are now ready to prove \cref{cor: hallin}.

\begin{proof}
For each $i\in [\alpha(G)]$, let $a_i = w(i,G)$. \cref{pro: Basic} implies $a_1 = \omega(G)$. Since $\rho(G) = \omega(G)$, \cref{pro: Basic} also implies $a_1 = \max\{a_i/i : i \in [\alpha(G)]\}$. Let $M = K_1 \vee G$, and note $\alpha(M)=\alpha(G)$. By \cref{thm: aTheorem}, for each $i\in [\alpha(G)]$, $w(i,M) = 1+a_i$. Then by \cref{pro: Basic}, $\rho(M) = \max\{(1+a_i)/i:i\in[\alpha(G)]\}$. For any $j \in [\alpha(G)]$, $a_1 \geq a_j/j$, which means $1+a_1 \geq (1+a_j)/j$. Thus, $\rho(M) = 1+a_1 = 1+\omega(G)$. 

Clearly $\omega(G) \leq \chi(G) \leq \chi_{\ell flex}(G) = k$. Since $\chi_{\ell flex}(G) = k$, $\epsilon_{\ell}(G,k) = 1/\rho(G) = 1/\omega(G)$. By \cref{thm: join}, $1/(\omega(G)+1) \leq \epsilon_{\ell}(M,k+1)$. Since $\rho(M) = 1+\omega(G)$, $\chi_{\ell flex}(K_1 \vee G) \leq k+1$.
\end{proof} 

Note that \cref{cor: hallin} doesn't apply when $G=C_{2k+1}$ and $k \geq 2$ since in this case $\omega(G)=2 < (2k+1)/k = \rho(G)$.  We will now show that the answer to \cref{ques: smallDADDY} is yes when $G$ is an odd cycle.

\begin{pro} \label{pro: oddcycles}
Suppose $k\geq 2$. Then $\chi_{\ell flex}(K_1 \vee C_{2k+1}) =4$. 
\end{pro}
\begin{proof}
Let $G=C_{2k+1}$. It is well known that $\alpha(G)=k$, $\rho(G) = (2k+1)/k$, and $\chi_{\ell flex}(G) = 3$ (see e.g.,~\cite{KM23}). Let $M = K_1 \vee G$. By \cref{thm: join}, $\epsilon_{\ell}(M,4) \geq k/(3k+1)$. Now we will show $\rho(M) =3$.

For each $i\in [\alpha(G)]$, let $a_i = w(i,G)$. Note that $a_1 = 2$ and $a_k = 2k+1$. Since $\rho(G) = (2k+1)/k$, \cref{pro: Basic} implies $a_k/k \geq a_j/j$ for each $j \in [\alpha(G)]$. By \cref{thm: aTheorem}, for each $i\in [\alpha(G)]$, $w(i,M) = 1+a_i$. Then by \cref{pro: Basic}, $\rho(M) = \max\{(1+a_i)/i:i\in[\alpha(G)]\}$. When $j$ satisfies $2 \leq j\leq k$, $1/j +1/k \leq 1$ which implies $$3 \geq \frac{2k+1}{k} +\frac{1}{j} = \frac{a_k}{k} + \frac{1}{j} \geq \frac{a_j+1}{j}. $$
Thus, $\rho(M) = \max\{(1+a_i)/i : i \in [\alpha(G)]\} = a_1 + 1 = 3$. 

We are now ready to prove the desired result. Suppose $L$ is an arbitrary $4$-assignment for $M$ and $r$ is a request of $L$ with domain $D$. Since $\epsilon_{\ell}(M,4) \geq k/(3k+1)$, there exists a proper $L$-coloring $f$ of $M$ satisfying at least $\lceil k|D|/(3k+1) \rceil$ of the vertex requests of $r$. Clearly, $0 < |D|\leq |V(M)| < 3k+1$ which means $0 < |D|/(3(3k+1)) < 1/3$. Thus, $$\left\lceil\frac{k}{3k+1}|D|\right\rceil = \left\lceil\frac{1}{3}|D|- \frac{1}{3(3k+1)}|D|\right\rceil = \left\lceil \frac{1}{3}|D|\right\rceil.$$  So, $f$ satisfies at least $\lceil |D|/3\rceil$ of the vertex requests of $r$.  Thus, $M$ is $(4, 1/\rho(M))$-flexible, and we have $4 = \chi(M) \leq \chi_{\ell flex}(M)$. Consequently, $\chi_{\ell flex}(M) = 4$.  
\end{proof}

It is worth mentioning that one can use \cref{cor: hallin} and \cref{pro: oddcycles} to give an easy proof by induction that $\chi_{\ell flex}(K_p \vee C_{2k+1}) = p+3 = \chi(K_p \vee C_{2k+1})$ whenever $p, k \in \N$ (when $k=1$ this result immediately follows from the fact that $\chi_{\ell flex}(K_n) = n$, see~\cite{KM23}). 

\section[Answering Question BIGDADDY]{Answering \cref{ques: BIGDADDY}} \label{answer}

In this section we show that the answer to \cref{ques: BIGDADDY} is no and characterize the graphs for which the answer is yes by proving \cref{thm: main}.  Before proving \cref{thm: main}, we need several technical lemmas.  Our first lemma demonstrates that the Hall ratio of the join of a complete graph and any graph $G$ is the clique number of the join, provided that the order of the complete graph is large enough.

\begin{lem} \label{pro: clique}
Suppose $G$ is a graph with two nonadjacent vertices (i.e., with $\alpha(G) \geq 2$).  For each $i\in [\alpha(G)]$, let $a_i = w(i,G)$.  Let $d_j = \lceil (a_j-ja_1)/(j-1) \rceil$ for each $j \in \{2, \dots, \alpha(G)\}$. If $N \geq \max \{1, d_2, \dots, d_{\alpha(G)} \}$, then $\omega(K_N \vee G) = \rho(K_N\vee G)$.
\end{lem}

\begin{proof}
\cref{pro: Basic} implies $\rho(G) = \max\{a_i/i : i \in [\alpha(G)]\}$.  Let $M = K_N \vee G$, and note that $\alpha(M)=\alpha(G)$. By \cref{thm: aTheorem}, for each $i\in [\alpha(G)]$, $w(i,M) = N+a_i$. Then, by \cref{pro: Basic}, $\rho(M) = \max\{(N+a_i)/i:i\in[\alpha(G)]\}$. Suppose $i\in \{2, \dots, \alpha(G)\}$.  We have $N \geq \lceil (a_i-ia_1)/(i-1) \rceil \geq (a_i-ia_1)/(i-1)$ which means $N+a_1 \geq (N+a_i)/i$. Thus, $\rho(M) = N+a_1 = \omega(M)$. 
\end{proof}

Note that if $G$ is a graph with $\alpha(G)=1$, then $\omega(K_N \vee G) = \rho(K_N \vee G)$ whenever $N \in \N$.  Using the notation of \cref{pro: clique}, from this point forward, for any graph $G$, we let $\tau_G = 1$ when $G$ is a complete graph and let $\tau_G = \max \{1, d_2, \dots, d_{\alpha(G)} \}$ otherwise.

 Our next lemma is well-known and involves list assignments for a graph $G$ that may assign lists of various sizes to the vertices of $G$.  If $f: V(G) \rightarrow \N$, we say that $G$ is \emph{$f$-choosable} if $G$ is $L$-colorable whenever $L$ is a list assignment for $G$ satisfying $|L(v)|=f(v)$ for each $v \in V(G)$.  If $f: V(G) \rightarrow \N$ and $g: V(G) \rightarrow \N$ satisfy $g(v) \geq f(v)$ for each $v \in V(G)$, it is immediately clear that $G$ is $f$-choosable only if $G$ is $g$-choosable.  We can now state the well-known ``Small Pot Lemma".

 \begin{lem} (Small Pot Lemma~\cite{CR12}) \label{lem: smallPot}
Suppose that $G$ is a graph and $f: V(G) \rightarrow \mathbb{N}$ is a function satisfying $f(v) < |V(G)|$ for all $v \in V(G)$.  If $G$ is not $f$-choosable, then there is an $f$-assignment $L$ for $G$ satisfying $|\bigcup_{v \in V(G)} L(v)| < |V(G)|$ such that there is no proper $L$-coloring of $G$.
\end{lem}

The context in which we apply the Small Pot Lemma first requires a straightforward observation.

\begin{obs}\label{lem: smallerPot}
Suppose $q,n,a \in \N$, $a \leq n$, and $U$ is a set with $|U| = n$. Suppose $A_1, \dots, A_q$ are subsets of $U$ each of size $a$. Then $$\left|{\bigcap}_{i=1}^{q} A_i \right| \geq n - (n-a)q .$$
\end{obs}

We now use \cref{lem: smallPot} and \cref{lem: smallerPot} to prove the last lemma we need for our proof of \cref{thm: main}.

\begin{lem}\label{lem: smallestPot}
Suppose $G$ is a $t$-vertex graph and $\chi(G) = k$. Suppose $S \subseteq V(G)$ is a nonempty subset of a color class $C'$ of some proper $k$-coloring of $G$. Let $C = C' - S$, $|S|=s$, and $H = K_N \vee (G-S)$ for some $N \in \N$. 
\\
(i)  If $C = \emptyset$ and $N \geq \max\{ 1,t-2k-s+1\} $, then $H$ is $(N+k-1)$-choosable.
\\
(ii) If $C \neq \emptyset$, $N \geq \max \{1, (t-(k+s))(t-(s+|C|+1)) \}$, and $f: V(H) \rightarrow \N$ is given by 
\begin{equation}
f(v) = \begin{cases}
N+k & \text{if } v \in C \\
N+k-1 & \text{otherwise},
\notag
\end{cases}
\end{equation}
then $H$ is $f$-choosable. 
\end{lem}

\begin{proof}
First, we prove Statement~(i). Note that Statement~(i) is clear when $k=1$; so, we may suppose $k \geq 2$.  Since $C = \emptyset$, $\chi(G-S) = k-1$. Also, $\chi(H) = N+k-1$ and $|V(H)|=N+t-s$. Since $N \geq t-2k - s + 1$, \cref{thm: Ohba} implies $H$ is $(N+k-1)$-choosable.

Now, we prove Statement~(ii).  First, we note that $k \geq 1$, and $|V(H)| \geq N+k$ since deleting the vertices in $S$ from $G$ leaves at least one vertex from each color class of a proper $k$-coloring of $G$.  Notice that if $k=1$ or $|V(H)|=N+k$, then one can greedily construct a proper $L$-coloring of $H$ whenever $L$ is an $f$-assignment for $H$ (by coloring the element(s) in $C$ last).  So, we may assume $|V(H)| > N+k$ and $k > 1$. 

For the sake of contradiction, suppose $H$ is not $f$-choosable. By \cref{lem: smallPot}, there is an $f$-assignment $L$ for $H$ satisfying $ |\bigcup_{v \in V(H)} L(v)| < |V(H)|=N+t-s$ such that there is no proper $L$-coloring of $H$. So, we may suppose $L(v) \subseteq [N+t-s-1]$ for each $v \in V(H)$.  Suppose $V(G-S)-C = \{u_1, \ldots, u_q \}$.  We let $r = |C|$ so that $q = t- r-s$.  Note that for each $j \in [q]$, $|L(u_j)| = N + k - 1$. By \cref{lem: smallerPot}, 
\begin{align*}
\left|{\bigcap}_{i=1}^{q} L(u_i) \right| &\geq N+t-s-1 - (t-s-k)(t-s-r) \\
& \geq (t-(k+s))(t-(s+r+1))- (t-(k+s))(t-(s+r)) + t - s - 1 =k-1. 
\end{align*}
Without loss of generality, we may assume that $[k-1] \subseteq {\bigcap}_{i=1}^{q} L(u_i)$. 

Let $L'(v) = L(v) - [k-1]$ for each $v \in W$ where $W$ is the vertex set of the copy of $K_N$ used to form $H$. Since $|L'(v)| \geq N$ for each $v \in W$, we can greedily construct a proper $L'$-coloring $g$ of $H[W]$. Color the elements of $W$ according to $g$. Now let $L''(v) = L(v) - g(W)$ for each $v \in V(G-S)$. Notice that for all $v \in (V(G-S)-C)$, $[k-1] \subseteq L''(v)$, and for all $v \in C$, $|L''(v)| \geq k$. Since $\chi(G-C')=k-1$, there is a proper $(k-1)$-coloring $h:V(G-C') \rightarrow [k-1]$ of $G-C'$.  Color the elements in $V(G-C')$ according to $h$. To complete a proper $L''$-coloring of $G-S$, color each $v \in C$ with a color that is in $L''(v)$ and isn't in $[k-1]$ (this maintains properness since $C$ is an independent set in $H$).  Thus, there exists a proper $L''$-coloring of $G-S$ which can be used to complete a proper $L$-coloring of $H$.  This however is a contradiction.
\end{proof}

We are now ready to prove \cref{thm: main}, which we restate.

\begin{customthm} {\bf \ref{thm: main}}
For every graph $G$, there is an $N\in\mathbb{N}$ such that
$\chi_{\ell flex}(K_p\vee G)=\chi(K_p\vee G)$ whenever $p\geq N$ if and only if, for every $S\subseteq V(G)$ with $|S|=\omega(G)+1$, there is a proper $\chi(G)$-coloring of $G$ that uses at most $\omega(G)=|S|-1$ colors on $S$.
\end{customthm}

\begin{proof}
Suppose $G$ is a graph such that for every $S \subseteq V(G)$ with $|S| = \omega(G) + 1$, there is a proper $\chi(G)$-coloring of $G$ that uses at most $\omega(G)$ colors on $S$.  Let $k = \chi(G)$ and $t = |V(G)|$. Suppose $N = \max \{2, \tau_G, (t-k-1)(t-3)+1, t - \omega(G)\}$ and $p \geq N$. Let $H = K_p \vee G$. Suppose the vertex set of the copy of $K_p$ used to form $H$ is $W = \{ w_1, \dots, w_p\}$.  Since $\chi(H) = k+p$, we must show $H$ is $(p+k, 1/ \rho(H))$-flexible.

Suppose $L$ is an arbitrary $(p+k)$-assignment for $H$ and $r$ is a request of $L$ with nonempty domain $D$. Since $p \geq \tau_G, \rho(H) = \omega(H) = \omega(G) + p$.  We must show that there is a proper $L$-coloring of $H$ that satisfies at least $|D|/\rho(H)$ of the vertex requests of $r$. We will demonstrate this in each of two possible cases: (1) $ |D| \leq \omega(G) + p$ and (2) $\omega(G) + p + 1 \leq |D| \leq |V(H)| = t + p$.  Note that if $\omega(G)=t$, then one need only consider (1).

In case (1), we will show that there is a proper $L$-coloring of $H$ satisfying at least $\lceil |D|/\rho(H) \rceil = 1$ vertex request of $r$. If there exists a $w \in D \cap W$, color $w$ with $r(w)$. For each $v \in V(H-w)$, let $L'(v) = L(v) - \{ r(w)\}$. Clearly, $|L'(v)| \geq p+k-1$. Note that $H-w = K_{p-1} \vee G$ and $\chi(H-w) = p + k - 1$. Since $p \geq t - \omega(G) \geq t -2k$, $H-w$ is chromatic-choosable by \cref{thm: Ohba}. Thus, there exists a proper $L'$-coloring of $H-w$ which can be used to complete a proper $L$-coloring of $H$ that colors $w$ with $r(w)$.

If $D$ and $W$ are disjoint, there is a $u \in V(G) \cap D$.  Color $u$ with $r(u)$. Let $C'$ be the color class containing $u$ in a proper $k$-coloring of $G$, and let $C = C' - \{u\}$. Then, for each $v \in (V(H-u)-C)$, let $L'(v) = L(v) - \{r(u)\}$, and for each $v \in C$, let $L'(v) = L(v)$. Clearly, $|L'(v)| \geq p+k-1$ for each $v \in V(H-u)$.  Moreover, $|L'(v)| = p+k$ for each $v \in C$. Note that if $C = \emptyset$, $p \geq \max\{2, t -\omega(G) \} \geq \max\{1,t-2k\}$, and \cref{lem: smallestPot}~(i) applies. Notice also that if $C \neq \emptyset$, then $p \geq \max\{2, (t-k-1)(t-3)+1\} \geq \max \{1, (t-k-1)(t-2-|C|) \}$, and  \cref{lem: smallestPot}~(ii) applies. Thus, in either case, there exists a proper $L'$-coloring of $H-u$ which can be used to complete a proper $L$-coloring of $H$ that colors $u$ with $r(u)$.

In case (2), since $p \geq t-\omega(G) \geq t-2\omega(G)$, we have $|V(H)| = p+t \leq 2(p+\omega(G)) = 2\rho(H)$. So, we must show there is a proper $L$-coloring of $H$ satisfying at least $\lceil |D|/\rho(H) \rceil = 2$ vertex requests of $r$. If there are $w_i,w_j \in W \cap D$ such that $r(w_i) \neq r(w_j)$, color $w_i$ with $r(w_i)$ and $w_j$ with $r(w_j)$.  For each $v \in V(H-\{w_i,w_j\})$, let $L'(v) = L(v) - \{r(w_i), r(w_j)\}$. Clearly, $|L'(v)| \geq p + k - 2$.  Note that $|V(H-\{w_i,w_j\})|=p+t-2$ and $\chi(H-\{w_i,w_j\})=p+k-2$. Since $p \geq t- \omega(G) \geq t-2k+1$, $H-\{w_i,w_j\}$ is chromatic-choosable by \cref{thm: Ohba}.  Thus, there exists a proper $L'$-coloring of $H-\{w_i,w_j\}$ which can be used to complete a proper $L$-coloring of $H$ that satisfies at least two vertex requests of $r$.

Now, we may assume $r$ is constant on $W \cap D$. Since $p \geq t-\omega(G)$, we know $|D| \geq p+\omega(G)+1 > t$ which means $W \cap D \neq \emptyset$. Suppose $w \in W \cap D$. Now, suppose that there is a $u \in D \cap V(G)$ such that $r(u) \neq r(w)$. Color $w$ with $r(w)$ and $u$ with $r(u)$. Note that $H-\{w,u\} = K_{p-1} \vee (G-u)$. Let $C'$ be the color class containing $u$ in a proper $k$-coloring of $G$, and let $C = C'-\{u\}$. For each $v \in (V(H-\{w,u\})-C)$, let $L'(v) = L(v) - \{r(w),r(u)\}$, and for each $v \in C$, let $L'(v) = L(v)-\{r(w)\}$. Clearly, $|L'(v)| \geq p+k-2$ for each $v \in V(H-\{w,u\})$.  Moreover, $|L'(v)| \geq p+k-1$ for each $v \in C$. Note that if $C = \emptyset$, $p-1 \geq \max\{1,t-\omega(G)-1\} \geq \max\{1,t-2k\}$, and \cref{lem: smallestPot}~(i) applies. Notice also that if $C \neq \emptyset$, $p-1 \geq \max\{1,(t-k-1)(t-3)\} \geq \max\{1, (t-k-1)(t-2-|C|)\}$, and \cref{lem: smallestPot}~(ii) applies. So, there exists a proper $L'$-coloring of $H-\{w,u\}$ which can be used to complete a proper $L$-coloring of $H$ that satisfies at least two vertex requests of $r$. 

If there is no $u \in D \cap V(G)$ with $r(u) \neq r(w)$, then $r$ is constant.  So, we may assume $r$ is constant and its range is $\{z\}$. Since $|D \cap V(G)| \geq \omega(G) + 1$, there must exist distinct $v_i,v_j \in D \cap V(G)$ that are colored the same in some proper $k$-coloring $f$ of $G$. Color $v_i$ and $v_j$ with $z$. Note that $H-\{v_i,v_j\}=K_p \vee (G- \{v_i,v_j\})$. Let $C'$ be the color class of $f$ containing $v_i$ and $v_j$, and let $C = C' - \{v_i,v_j\}$. For each $v \in (V(H-\{v_i, v_j\})-C)$, let $L'(v) = L(v) - \{z\}$, and for each $v \in C$, let $L'(v) = L(v)$. Clearly, $|L'(v)| \geq p+k-1$ for all $v \in V(H-\{v_i,v_j\})$. Moreover, $|L'(v)| = p+k$ whenever $v \in C$. Note that if $C = \emptyset$, $p \geq \max\{2,t-\omega(G)\} \geq \max\{1,t-2k-1\}$, and \cref{lem: smallestPot}~(i) applies. Notice also that if $C \neq \emptyset$, then $p \geq \max \{2,(t-k-1)(t-3)+1\}\geq \max\{1, (t-k-2)(t-3-|C|) \}$, and \cref{lem: smallestPot}~(ii) applies. So, there exists a proper $L'$-coloring of $H-\{v_i,v_j\}$ which can be used to complete a proper $L$-coloring of $H$ that satisfies at least two vertex requests of $r$. Having completed cases (1) and (2), we may conclude $\chi_{\ell flex}(K_p \vee G) =k+p= \chi(K_p \vee G)$.

Conversely, suppose $G$ is a graph such that there is an $S \subseteq V(G)$ with $|S| = \omega(G) + 1$ for which every proper $\chi(G)$-coloring of $G$ uses $\omega(G)+1$ colors on $S$.  Let $k = \chi(G)$. Suppose $p \geq \tau_G$, and $H = K_p \vee G$. Notice that since $p \geq \tau_G, \rho(H) = \omega(H) = \omega(G) + p$. We claim $\chi_{\ell flex}(H) > k+p = \chi(H)$. 

Suppose the vertex set of the copy of $K_p$ used to form $H$ is $W$. Let $L$ be the $(p+k)$-assignment for $H$ given by $L(v) = [p+k]$ for each $v \in V(H)$. Suppose $r : W \cup S \rightarrow [p+k]$ is the request of $L$ given by $r(v) = 1$. To prove the desired claim, we will show that there is no proper $L$-coloring of $H$ that satisfies at least $|W \cup S|/\rho(H)$ of the vertex requests of $r$.

For the sake of contradiction, suppose $f$ is a proper $L$-coloring of $H$ that satisfies at least $\lceil |W \cup S|/\rho(H) \rceil = \lceil (p+\omega(G) + 1)/(\omega(G) + p) \rceil = 2$ vertex requests of $r$. Note $f$ can't color any vertex in $W$ with 1 for if $f$ colors some vertex in $W$ with $1$, then $|f^{-1}(1)|=1$ which contradicts the fact that $f$ satisfies two vertex requests of $r$.

Now, let $A = f(W)$.  Notice $1 \notin A$, $|A| = p$, and $|[p+k]-A| = k$. Next, let $g$ be $f$ with domain restricted to $V(G)$. Clearly, $g$ is a proper $k$-coloring of $G$ that uses the colors in the set $([p+k]-A)$.  Since $f$ satisfies at least $2$ vertex requests of $r$ and colors no vertices in $W$ with $1$, $g$ is a proper $k$-coloring of $G$ that colors two vertices of $S$ with $1$. This is a contradiction. Thus, we have produced a request $r$ of $L$ for which there is no proper $L$-coloring satisfying an appropriate number of vertex requests of $r$. Consequently, $\chi_{\ell flex}(H) > k+p$. 
\end{proof}

{\bf Acknowledgment.} The authors would like to thank Hemanshu Kaul for his helpful comments.  This project was completed by the Coloring Research Group of South Alabama at the University of South Alabama during the spring 2025, summer 2025, and fall 2025 semesters.  The support of the University of South Alabama is gratefully acknowledged.

	\bibliographystyle{hplain}
\bibliography{bibliography}{}
\vspace{-5mm}

\end{document}